\documentclass[11pt]{article}
\usepackage[a4paper,margin=1in]{geometry}
\usepackage[T1]{fontenc}
\usepackage[utf8]{inputenc}
\usepackage{lmodern}
\usepackage{microtype}
\usepackage{amsmath,amssymb,amsthm,mathtools}
\usepackage{enumitem,booktabs,array}
\usepackage{graphicx}
\usepackage{tikz}
\usetikzlibrary{positioning,calc,fit,arrows.meta,backgrounds,shadows.blur}
\usepackage{hyperref}
\usepackage{aliascnt}
\usepackage[nameinlink,noabbrev]{cleveref}
\usepackage{xcolor}
\definecolor{aeedge}{RGB}{70,0,70}
\definecolor{aefill}{RGB}{243,230,247}
\definecolor{fedge}{RGB}{0,0,166}
\definecolor{ffill}{RGB}{226,232,253}
\definecolor{aledge}{RGB}{0,97,72}
\definecolor{alfill}{RGB}{216,237,229}
\definecolor{lledge}{RGB}{200,69,0}
\definecolor{llfill}{RGB}{254,236,213}
\definecolor{binedge}{RGB}{144,95,19}
\definecolor{lincolor}{RGB}{189,121,7}

\hypersetup{
  colorlinks=true,
  linkcolor=blue!50!black,
  citecolor=green!40!black,
  urlcolor=blue!60!black,
  pdftitle={Folded-Algebraic Matroids: Characteristic Rigidity and Almost-Entropic Separation},
  pdfauthor={Shahram Khazaei},
  pdfsubject={Matroid representation and almost-entropic separation},
  pdfkeywords={algebraic matroid, multilinear matroid, characteristic set, group configuration, almost-entropic matroid}
}
\setlist{itemsep=2pt,topsep=4pt}
\newtheorem{theorem}{Theorem}[section]
\newaliascnt{lemma}{theorem}\newtheorem{lemma}[lemma]{Lemma}\aliascntresetthe{lemma}
\newaliascnt{proposition}{theorem}\newtheorem{proposition}[proposition]{Proposition}\aliascntresetthe{proposition}
\newaliascnt{corollary}{theorem}\newtheorem{corollary}[corollary]{Corollary}\aliascntresetthe{corollary}
\theoremstyle{definition}
\newaliascnt{definition}{theorem}\newtheorem{definition}[definition]{Definition}\aliascntresetthe{definition}
\newaliascnt{remark}{theorem}\aliascntresetthe{remark}

\crefname{theorem}{theorem}{theorems}\Crefname{theorem}{Theorem}{Theorems}
\crefname{lemma}{lemma}{lemmas}\Crefname{lemma}{Lemma}{Lemmas}
\crefname{proposition}{proposition}{propositions}\Crefname{proposition}{Proposition}{Propositions}
\crefname{corollary}{corollary}{corollaries}\Crefname{corollary}{Corollary}{Corollaries}
\crefname{definition}{definition}{definitions}\Crefname{definition}{Definition}{Definitions}
\crefname{remark}{remark}{remarks}\Crefname{remark}{Remark}{Remarks}

\newcommand{\acl}{\operatorname{acl}}
\newcommand{\trdeg}{\operatorname{trdeg}}
\newcommand{\cl}{\operatorname{cl}}
\newcommand{\FAlg}{\mathrm{FAlg}}
\newcommand{\FLin}{\mathrm{FLin}}
\newcommand{\Alg}{\mathrm{Alg}}
\newcommand{\Lin}{\mathrm{Lin}}
\newcommand{\AlmostEnt}{\mathrm{AE}}

\newcommand{\ind}{\mathrel{\perp\!\!\!\perp}}
\newcommand{\simF}{\mathrel{\sim_F}}

\title{\textbf{Folded-Algebraic Matroids: Characteristic Rigidity and Almost-Entropic Separation}}
\author{Shahram Khazaei\\
\small Department of Mathematical Sciences, Sharif University of Technology, Tehran, Iran\\
\small \href{mailto:shahram.khazaei@sharif.ir}{shahram.khazaei@sharif.ir}\\
\small ORCID: \href{https://orcid.org/0000-0002-2493-8840}{0000-0002-2493-8840}}
\date{}

\begin{document}
\maketitle

\begin{abstract}
We introduce \emph{folded-algebraic matroids}.  In such a representation, every matroid element is replaced by a finite tuple of algebraic quantities, and transcendence degree agrees with matroid rank after one uniform scaling.  The resulting class contains both algebraic and folded-linear matroids and is contained in the class of almost-entropic matroids, whose rank functions are limits of scaled entropy functions.  We prove that the latter containment is proper.  Our main result concerns the classical rank-three matroids $M(p)$ of Gordon.  For every prime $p$, we show that $M(p)$ has a folded-algebraic representation over a field $K$ if and only if $K$ has characteristic $p$.  We then use a point-identification construction that preserves almost-entropicity to obtain a $13$-element rank-three $3$-connected matroid $C_{2,3}$ that is almost entropic but not folded algebraic.  Choosing a common element as dealer also yields a connected $12$-participant port with incompatible characteristic requirements.  Finally, we record compact explicit witnesses and size bounds for several other separating regions among the representation classes.
\end{abstract}



\section{Introduction}\label{sec:intro}

Representability is one of the central themes in matroid theory.  A matroid records an abstract pattern of dependence, and a representation realizes that pattern inside a more concrete mathematical structure.  In a linear representation, matroid elements are vectors over a field and matroid independence agrees with linear independence.  In an algebraic representation, elements are field quantities and rank is measured by transcendence degree.  A third exact model, usually called multilinear or folded linear, replaces each matroid element by a finite-dimensional vector space of one common dimension.

The algebraic and folded-linear classes are incomparable~\cite{BenEfraim2016,Matus1999}.  This makes a common enlargement natural.  Let $M$ be a matroid on ground set $E$, with rank function $r_M$, and fix a base field $K$.  Instead of assigning one algebraic quantity to each element $e\in E$, assign a finite algebraic \emph{packet} $X_e$ in a field extension of $K$.  For $A\subseteq E$, put $X_A=\bigcup_{e\in A}X_e$.  If every packet has transcendence degree $t$ over $K$ and
\[
\trdeg_K K(X_A)=t\,r_M(A)
\]
for every $A\subseteq E$, then the packet representation reproduces the matroid rank after a uniform scaling by $t$.  We call such representations \emph{folded algebraic}.  Ordinary algebraic representations are the case $t=1$, while folded-linear subspace arrangements yield folded-algebraic representations by evaluating their vectors as linear forms in independent transcendentals.

Entropy leads to a larger approximation class.  The joint entropies of a finite family of random variables form a monotone submodular rank function; limits of positive scalar multiples of such functions are called \emph{almost entropic}.  Mat\'u\v{s} proved that every algebraic matroid is almost entropic~\cite{Matus2024}.  Replacing each algebraic packet by scalar coordinates gives an ordinary algebraic matroid on a partitioned ground set, and grouping the coordinates back into packets gives
\[
\Alg\cup\FLin\subseteq\FAlg\subseteq\AlmostEnt.
\]
The landscape is summarized in \cref{fig:venn}.  We use $F_7$ for the Fano plane and $F_7^-$ for the non-Fano matroid obtained by relaxing one of its lines.  The symbols $Q^{\mathrm{NF}}_3(\mathsf Q_8)$ and $R_7$ denote the two components of Ben-Efraim's characteristic-$7$ construction, while $NnF(7)$ denotes the generalized non-Fano matroid of Sharififar--Sadeghi--Aboutorab~\cite{BenEfraim2016,SharififarEtAl2024}.  For matroids on disjoint ground sets, $M_1\oplus M_2$ denotes their direct sum, whose rank is the sum of the two component ranks.  The non-Pappus matroid shows that the overlap is strictly larger than the linear class:
\[
\Lin\subsetneq \Alg\cap\FLin.
\]
Indeed, the non-Pappus matroid is algebraic over finite fields~\cite{LindstromNonPappus1986}, has a $2$-fold linear representation over $\mathbb Q$~\cite[Example~4.2]{PendavinghVanZwam2013}, and is not linearly representable over any field.  The integral block matrix in the cited example also gives a $2$-fold linear representation over $\mathbb F_5$ after reduction modulo $5$.

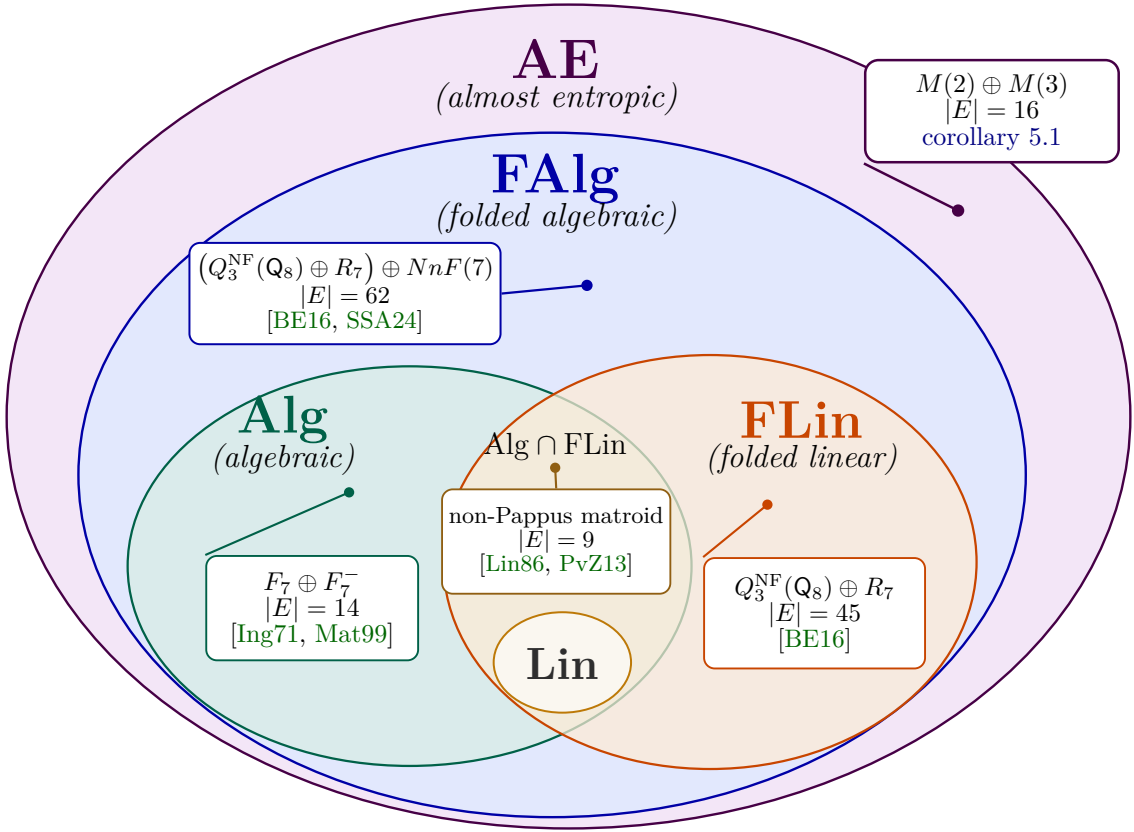
\begin{figure}[!ht]
\centering
\begin{tikzpicture}[x=.01355cm,y=-.01355cm]
\useasboundingbox (0,0) rectangle (1169,835);

\path[fill=aefill,draw=aeedge,line width=.9pt] (565,408) ellipse [x radius=548,y radius=402];
\path[fill=ffill,draw=fedge,line width=.9pt] (549,465) ellipse [x radius=462,y radius=334];
\path[fill=alfill,fill opacity=.72,draw=aledge,line width=.85pt] (410,554) ellipse [x radius=275,y radius=195];
\path[fill=llfill,fill opacity=.72,draw=lledge,line width=.85pt] (703,550) ellipse [x radius=260,y radius=202];
\path[fill=white,fill opacity=.60,draw=lincolor,line width=.75pt] (559,648) ellipse [x radius=67,y radius=49];

\node[text=aeedge,font=\bfseries\fontsize{21}{22}\selectfont] at (552,59) {AE};
\node[font=\itshape\fontsize{12.5}{13.5}\selectfont] at (552,98) {(almost entropic)};
\node[text=fedge,font=\bfseries\fontsize{21}{22}\selectfont] at (552,176) {FAlg};
\node[font=\itshape\fontsize{12.5}{13.5}\selectfont] at (552,214) {(folded algebraic)};
\node[text=aledge,font=\bfseries\fontsize{20}{21}\selectfont] at (286,411) {Alg};
\node[font=\itshape\fontsize{12}{13}\selectfont] at (286,449) {(algebraic)};
\node[text=lledge,font=\bfseries\fontsize{20}{21}\selectfont] at (792,411) {FLin};
\node[font=\itshape\fontsize{12}{13}\selectfont] at (792,449) {(folded linear)};
\node[font=\fontsize{11.5}{12.5}\selectfont] at (552,437) {$\Alg\cap\FLin$};
\node[text=black!82,font=\bfseries\fontsize{17}{18}\selectfont] at (559,649) {Lin};

\node[draw=fedge,fill=white,rounded corners=4pt,line width=.75pt,
      minimum width=3.35cm,minimum height=1.34cm,align=center,inner sep=2.5pt,
      font=\fontsize{9.0}{10.2}\selectfont] (falgfirst) at (347,288)
 {$\bigl(Q^{\mathrm{NF}}_3(\mathsf Q_8)\oplus R_7\bigr)\oplus NnF(7)$\\[-1pt]
  $|E|=62$\\[-1pt]
  \cite{BenEfraim2016,SharififarEtAl2024}};
\coordinate (falgpoint) at (583,280);
\fill[fedge] (falgpoint) circle [radius=5.5];
\draw[fedge,line width=.85pt] (falgfirst.east) -- (falgpoint);

\node[draw=aledge,fill=white,rounded corners=4pt,line width=.72pt,
      minimum width=2.80cm,minimum height=1.38cm,align=center,inner sep=3pt,
      font=\fontsize{9.7}{10.9}\selectfont] (algfirst) at (315,595)
 {$F_7\oplus F_7^-$\\[-1pt]
  $|E|=14$\\[-1pt]
  \cite{Ingleton1971,Matus1999}};
\coordinate (algpoint) at (351,482);
\fill[aledge] (algpoint) circle [radius=5.0];
\draw[aledge,line width=.8pt] (algpoint) -- (algfirst.north west);

\node[draw=binedge,fill=white,rounded corners=4pt,line width=.70pt,
      minimum width=2.23cm,minimum height=1.38cm,align=center,inner sep=2.5pt,
      font=\fontsize{8.8}{10.0}\selectfont] (bothfirst) at (553,530)
 {non-Pappus matroid\\[-1pt]
  $|E|=9$\\[-1pt]
  \cite{LindstromNonPappus1986,PendavinghVanZwam2013}};
\coordinate (bothpoint) at (552,458);
\fill[binedge] (bothpoint) circle [radius=5.0];
\draw[binedge,line width=.8pt] (bothpoint) -- (bothfirst.north);

\node[draw=lledge,fill=white,rounded corners=4pt,line width=.72pt,
      minimum width=2.92cm,minimum height=1.46cm,align=center,inner sep=3pt,
      font=\fontsize{9.5}{10.7}\selectfont] (flinfirst) at (805,600)
 {$Q^{\mathrm{NF}}_{3}(\mathsf Q_{8})\oplus R_{7}$\\[-1pt]
  $|E|=45$\\[-1pt]
  \cite{BenEfraim2016}};
\coordinate (flinpoint) at (759,494);
\fill[lledge] (flinpoint) circle [radius=5.0];
\draw[lledge,line width=.8pt] (flinpoint) -- (flinfirst.north west);

\node[draw=aeedge,fill=white,rounded corners=5pt,line width=.95pt,
      minimum width=3.35cm,minimum height=1.35cm,align=center,inner sep=3pt,
      font=\fontsize{9.7}{10.9}\selectfont] (aefirst) at (978,110)
 {$M(2)\oplus M(3)$\\[-1pt]
  $|E|=16$\\[-1pt]
  \cref{cor:direct-separation}};
\coordinate (aepoint) at (945,207);
\fill[aeedge] (aepoint) circle [radius=6.0];
\draw[aeedge,line width=.95pt] (aefirst.south west) -- (aepoint);
\end{tikzpicture}
\caption{The classes $\Lin$, $\Alg$, $\FLin$, $\FAlg$, and $\AlmostEnt$ of linear, algebraic, folded-linear, folded-algebraic, and almost-entropic matroids.  The layout is adapted from Bamiloshin--Ben-Efraim--Farr\`as--Padr\'o~\cite[Fig.~1]{BamiloshinEtAl2021}.  Each box gives the first explicit witness for the indicated region, together with its ground-set size and source.  The corresponding improvements in size are recorded in \cref{tab:witness-sizes}.}
\label{fig:venn}
\end{figure}

For a nonempty class $\mathcal C$ of matroids, write
\[
n_{\min}(\mathcal C)=\min\{|E(M)|:M\in\mathcal C\}.
\]
The known values and intervals for the five separating regions in \cref{fig:venn} are collected in \cref{tab:witness-sizes}.

\begin{table}[t]
\centering
\caption{Historical improvement of upper bounds for the separating regions in \cref{fig:venn}.  The second column lists the first explicit witness, and the third records later improvements, with section references for the constructions proved here.  In the last row both the first witness and its improvement appear in the third column, separated from the inner regions by a horizontal rule.  Every matroid on at most seven elements is linear over some field; the lower bounds and the relevant eight-element analysis are given in~\cite{BamiloshinEtAl2021}.  The notation $P_1,P'_2,P''_2,P_3,L'_8$ follows that reference.}
\label{tab:witness-sizes}
\begingroup
\scriptsize
\setlength{\tabcolsep}{2.6pt}
\renewcommand{\arraystretch}{1.16}
\begin{tabular}{@{}>{\raggedright\arraybackslash}p{0.205\linewidth}>{\raggedright\arraybackslash}p{0.245\linewidth}>{\raggedright\arraybackslash}p{0.34\linewidth}>{\centering\arraybackslash}p{0.17\linewidth}@{}}
\toprule
Region & First published explicit witness & Upper-bound history & Certified interval\\
\midrule
$(\Alg\cap\FLin)\setminus\Lin$
 & non-Pappus, $9$ elements~\cite{LindstromNonPappus1986,PendavinghVanZwam2013}
 & $9\to8$ via the algebraic folded-linear matroids $P_3$ and $L'_8$~\cite[Propositions~4.1 and~4.5]{BamiloshinEtAl2021}
 & $n_{\min}=8$\\
$\Alg\setminus\FLin$
 & $F_7\oplus F_7^-$, $14$ elements~\cite{Ingleton1971,Matus1999}
 & $14\to10$ via the rank-$4$ matroid $M^{\mathrm{FN}}_{10}$; \cref{prop:mfn10}
 & $8\le n_{\min}\le10$\\
$\FLin\setminus\Alg$
 & $Q^{\mathrm{NF}}_3(\mathsf Q_8)\oplus R_7$, $45$ elements~\cite{BenEfraim2016}
 & $45\to38$~\cite[Remark~2]{BenEfraim2016}$\to24$ via the rank-$4$ matroid $A_{24}^{(5)}$; \cref{subsec:m33}
 & $9\le n_{\min}\le24$\\
$\FAlg\setminus(\Alg\cup\FLin)$
 & characteristic-$7$ bridge, $62$ elements~\cite{BenEfraim2016,SharififarEtAl2024}
 & $62\to33$ via the rank-$6$ matroid $M_{33}$; \cref{subsec:m33}
 & $8\le n_{\min}\le33$\\
\midrule
$\AlmostEnt\setminus\FAlg$
 &
 & $M(2)\oplus M(3)$, $16$ elements; \cref{cor:direct-separation}; $16\to13$ via the rank-$3$ matroid $C_{2,3}$; \cref{subsec:c23}
 & $8\le n_{\min}\le13$\\
\bottomrule
\end{tabular}
\endgroup
\end{table}

The fourth row of \cref{tab:witness-sizes} begins with the $62$-element bridge
\[
\bigl(Q^{\mathrm{NF}}_3(\mathsf Q_8)\oplus R_7\bigr)\oplus NnF(7).
\]
The first factor is folded linear but non-algebraic and forces folded-linear characteristic $7$~\cite{BenEfraim2016}; the second is algebraic in characteristic $7$ but has no folded-linear representation in that characteristic~\cite{SharififarEtAl2024}.  Their direct sum is therefore folded algebraic but neither algebraic nor folded linear.  The first interval in the table is known exactly: $n_{\min}((\Alg\cap\FLin)\setminus\Lin)=8$.  The other four entries remain intervals.

On the algebraic-only side, $M^{\mathrm{FN}}_{10}$ is a connected rank-four matroid on ten elements.  It contains the Fano plane and the non-Fano matroid as minors, which impose incompatible conditions on any folded-linear representation; its algebraic representation is given in \cref{prop:mfn10}.

The outer inclusion in \cref{fig:venn} raises the main question of this paper: does finite algebraic packetization already capture every almost-entropic matroid?  Characteristic rigidity first gives the disconnected $16$-element witness $M(2)\oplus M(3)$.  Identifying three common independent points then produces the connected $13$-element witness $C_{2,3}$.  It is $3$-connected, meaning that it has no separation of order one or two.

\begin{theorem}[Connected strict separation]\label{thm:intro-strict}
\[
\FAlg\subsetneq\AlmostEnt.
\]
More explicitly, there is a simple rank-three matroid $C_{2,3}$ on $13$ elements such that
\[
C_{2,3}\in\AlmostEnt\setminus\FAlg
\qquad\text{and}\qquad
C_{2,3}\text{ is $3$-connected}.
\]
\end{theorem}

The main ingredient is a characteristic theorem for the classical rank-three family $M(p)$ introduced by Gordon and studied subsequently by Lindstr\"om and Fl\'orez~\cite{Gordon1988,Lindstrom1985,Florez2006}.  Ordinary algebraic representations of $M(p)$ are known to force characteristic $p$.  We show that this obstruction survives arbitrary finite packet size.

For a matroid $N$, let $\chi_{\FAlg}(N)$ denote the set of characteristics of fields over which $N$ has a folded-algebraic representation.

\begin{theorem}[Folding-stable characteristic rigidity]\label{thm:intro-singleton}
For every prime $p$,
\[
\chi_{\FAlg}(M(p))=\{p\}.
\]
Equivalently, $M(p)$ has a folded-algebraic representation over a field $K$ if and only if $\operatorname{char}K=p$.
\end{theorem}

The statement is uniform over all folding degrees.  Classical one-coordinate methods based on derivations, tangent spaces, or Frobenius twists produce linear shadows of ordinary algebraic representations~\cite{BollenDraismaPendavingh2018,CartwrightVarghese2024}, but they do not reduce a positive-dimensional packet representation to an ordinary algebraic representation.  Moreover, a local configuration need not determine the characteristic: the non-Fano matroid, for example, has a monomial algebraic realization over every field.  The obstruction in $M(p)$ must therefore be obtained from the global arrangement of its local configurations.

We briefly describe the proof.  A folded representation of degree $t$ records the matroid inside field-theoretic algebraic closure, with every matroid point represented by an algebraically closed packet of transcendence degree $t$.  A group configuration is a seven-point incidence pattern from which an algebraic group can be recovered; the packet size becomes the dimension of that group.  Each seven-point cell of $M(p)$ is such an abelian configuration.  All cells contain the same triple $(X,A,Y)$, so every recovered group is compared directly with one reference group and the finite indeterminacies are removed by a common quotient.  Translation by a point defined over the base field leaves algebraic-closure classes unchanged.  Write $G(F)$ for the group of $F$-rational points.  Consequently, if $a$ is the common step and $d_i$ is the translated coordinate in the $i$th cell, then
\[
d_{i+1}=a+d_i+\varepsilon_i,\qquad \varepsilon_i\in G(F).
\]
Here $[p]$ denotes multiplication by $p$, or equivalently $p$-fold addition.  Closing the $p$-cycle gives $[p]a\in G(F)$.  Since $a$ is generic in a positive-dimensional connected group, this is a torsion identity for the whole group: it forces $[p]G=0$.  The differential of multiplication by $p$ then forces the field characteristic to be $p$.

Folded-algebraic characteristic sets intersect under direct sums, so distinct singleton spectra are incompatible in one field.  Almost-entropic rank functions are closed under independent products.  Therefore $M(p)\oplus M(q)$ is almost entropic but not folded algebraic whenever $p\ne q$.

The direct-sum witness is disconnected.  To obtain a connected example, we prove that almost-entropic matroids admit an \emph{independent-set welding}: matched independent points are identified by freely adjoining a point on their rank-two flat, contracting the new point, and deleting one of the matched points.  Repeating this operation preserves both factors as restrictions.  Welding the common bases
\[
\{A,X,S_0\}\subseteq M(2),M(3)
\]
produces a rank-three matroid $C_{2,3}$ on $7+9-3=13$ elements.  It is almost entropic by construction and cannot be folded algebraic because its two spanning restrictions force characteristics $2$ and $3$.  The spanning Fano restriction rules out every $2$-separation, so $C_{2,3}$ is $3$-connected.  Thus the strict containment is witnessed without a direct-sum decomposition.

The common elements also give a single-dealer consequence.  For a chosen dealer element $d$, the matroid port consists of those participant sets whose closure contains $d$.  Taking the shared point $A$ as dealer gives a connected $12$-participant port containing the $A$-ports of both $M(2)$ and $M(3)$ as deletion subports.  Every participant is essential, meaning that it belongs to a minimal set whose closure contains the dealer.  The Brickell--Davenport rigidity theorem then shows that the port has no exact folded-algebraic realization in which the dealer and the participants are represented by packets of one common size.

The paper is organized as follows.  \Cref{sec:classes} introduces the matroid terminology and representation classes used throughout, then gives the compact witness $M^{\mathrm{FN}}_{10}\in\Alg\setminus\FLin$.  \Cref{sec:Mp} describes the classical matroids $M(p)$ and their seven-point cells, gives the $33$-element folded-algebraic-only witness, and records the accompanying computational certificates.  These exact finite calculations are supplied separately in a computational companion archived at Zenodo~\cite{KhazaeiCompanion2026}.  \Cref{sec:proof} proves \cref{thm:intro-singleton}.  \Cref{sec:separation} establishes the direct-sum consequence, develops almost-entropic independent-set welding, and proves the connected separation and its single-dealer corollary.

\section{Representation classes and background}\label{sec:classes}

\subsection{Matroid terminology and classical examples}\label{subsec:matroid-background}

We use matroids through their rank functions.  Thus a matroid $M$ consists of a finite ground set $E=E(M)$ and an integer-valued, monotone, submodular rank function $r_M:2^E\to\mathbb Z_{\ge0}$ satisfying $r_M(\varnothing)=0$ and $r_M(A)\le |A|$.  Its rank is $r(M)=r_M(E)$.  For $A\subseteq E$ and $e\in E$, write $Ae=A\cup\{e\}$.  The closure of $A$ is
\[
\cl_M(A)=\{e\in E:r_M(Ae)=r_M(A)\}.
\]
A set is \emph{independent} if $r_M(A)=|A|$, and a \emph{circuit} is a minimal dependent set.  A matroid is \emph{simple} if every set of at most two elements is independent, and it is \emph{connected} if every two elements lie in a common circuit.  A \emph{flat} is a set equal to its closure, and a \emph{hyperplane} is a maximal proper flat.  A rank-$r$ matroid is \emph{sparse paving} if every set of size less than $r$ is independent and every dependent $r$-set is a hyperplane.

A \emph{$2$-separation} of a connected matroid is a partition $(U,V)$ of its ground set such that $|U|,|V|\ge2$ and $r_M(U)+r_M(V)=r(M)+1$.  A connected matroid with at least four elements is \emph{$3$-connected} if it has no $2$-separation.

Deletion and contraction of $e\in E$ are denoted by $M\setminus e$ and $M/e$.  Their rank functions on subsets of $E-\{e\}$ are
\[
r_{M\setminus e}(A)=r_M(A),
\qquad
r_{M/e}(A)=r_M(Ae)-r_M(\{e\}).
\]
A \emph{minor} is obtained by a sequence of deletions and contractions.  If $M_1$ and $M_2$ have disjoint ground sets, their direct sum $M_1\oplus M_2$ has rank
\[
r_{M_1\oplus M_2}(A)=r_{M_1}(A\cap E(M_1))+r_{M_2}(A\cap E(M_2)).
\]
For $S\subseteq E(M)$, the \emph{restriction} $M|S$ is obtained by deleting $E(M)-S$.  The set $S$ is \emph{spanning} if $r_M(S)=r(M)$.

The Fano plane $F_7$ is the rank-three binary matroid on the seven nonzero vectors of $\mathbb F_2^3$.  The non-Fano matroid $F_7^-$ is obtained by relaxing one circuit-hyperplane of $F_7$.  Thus $F_7$ is linearly representable precisely in characteristic $2$, whereas $F_7^-$ is linearly representable precisely outside characteristic $2$.

\subsection{Representation classes}\label{subsec:representation-classes}

\begin{definition}[Linear, algebraic, and folded-linear representations]\label{def:classical}
Let $M=(E,r)$.
\begin{enumerate}[label=(\roman*)]
\item $M$ is \emph{linear over $K$} if there are vectors $v_e$ such that $r(A)=\dim_K\langle v_e:e\in A\rangle$.
\item $M$ is \emph{algebraic over $K$} if there are field elements $x_e$ in an extension of $K$ such that
\[
r(A)=\trdeg_K K(x_e:e\in A).
\]
\item $M$ is \emph{$t$-fold linear over $K$} if there are $t$-dimensional subspaces $V_e$ such that
\[
\dim_K\sum_{e\in A}V_e=t\,r(A).
\]
It is \emph{folded linear} if this holds for some $t\ge1$ and some field.
\end{enumerate}
We denote the corresponding classes by $\Lin,\Alg,\FLin$.
\end{definition}

\begin{definition}[Folded algebraic]\label{def:falg}
For $t\ge1$, $M$ is \emph{$t$-fold algebraic over $K$} if there are finite tuples $X_e$ in an extension of $K$ such that
\begin{equation}\label{eq:falg}
\trdeg_K K(X_e:e\in A)=t\,r_M(A)\qquad(A\subseteq E).
\end{equation}
It is \emph{folded algebraic} if this holds for some $t$ and some field.  We write $\FAlg$ for the class.
\end{definition}

A useful equivalent form replaces each tuple by a transcendence basis of size $t$.

\begin{lemma}[Block lift]\label{lem:blocklift}
A matroid is $t$-fold algebraic over $K$ if and only if there is an ordinary algebraic matroid $N$ whose ground set is partitioned into blocks $X_e$ of size $t$ and
\[
r_N(X_A)=t\,r_M(A),
\qquad X_A:=\bigcup_{e\in A}X_e,
\qquad(A\subseteq E).
\]
\end{lemma}
\begin{proof}
For each nonloop, replace its representing tuple by a transcendence basis $X_e$ of the field generated by that tuple over $K$.  The compositum generated by the original tuples is algebraic over that generated by the bases, so all transcendence degrees are unchanged.  Conversely, regard each block $X_e$ as one tuple.
\end{proof}

Ordinary algebraicity is the case $t=1$, so $\Alg\subseteq\FAlg$.  Also $\FLin\subseteq\FAlg$: if $V_e$ is a $t$-dimensional subspace arrangement in $K^m$ and $z_1,\dots,z_m$ are algebraically independent, choose a basis $v_{e1},\dots,v_{et}$ of $V_e$ and represent $e$ by the $t$ linear forms $v_{ej}\cdot(z_1,\dots,z_m)$.  Their transcendence degree on every union of blocks equals the dimension of the corresponding span.

The two subclasses are known to be incomparable; see Ben-Efraim~\cite{BenEfraim2016} and the references discussed there.

\subsection{A ten-element algebraic-only witness}\label{subsec:mfn10}

The following partition model supplies a convenient obstruction to folded linearity.

\begin{definition}[Partition representation]\label{def:partition}
A matroid $(E,r)$ is \emph{partition representable of degree $d\ge2$} if there are partitions $\xi_e$ of a finite set $\Omega$, with $|\Omega|=d^{r(E)}$, such that for every $I\subseteq E$ the common refinement $\bigwedge_{e\in I}\xi_e$ has $d^{r(I)}$ blocks of equal size~\cite{Matus1999}.
\end{definition}

\begin{proposition}[A ten-element algebraic-only witness]\label{prop:mfn10}
There is a simple connected rank-four sparse-paving matroid $M^{\mathrm{FN}}_{10}$ on ten elements such that
\[
M^{\mathrm{FN}}_{10}\in\Alg\setminus\FLin.
\]
More precisely, $M^{\mathrm{FN}}_{10}$ is algebraic in characteristic $2$ and is not partition representable.
\end{proposition}
\begin{proof}
Let $E=\{0,1,\ldots,9\}$, and let the circuit-hyperplanes be
\begin{equation}\label{eq:mfn-ch}
\begin{gathered}
0128,\ 0147,\ 0169,\ 0257,\ 0367,\ 1249,\\
1267,\ 1357,\ 1468,\ 2347,\ 4567.
\end{gathered}
\end{equation}
Here, for example, $0128$ denotes the set $\{0,1,2,8\}$; the other strings use the same shorthand.
Thus every set of size at most three is independent, the eleven displayed four-sets have rank three, and every other four-set is a basis.  The circuit-hyperplanes meet pairwise in at most two elements.  Consequently, a five-set cannot contain two circuit-hyperplanes, so it contains a four-element basis and has rank four; the same is then true of every larger set.  Over $K=\mathbb F_2(x,y,z,t)$ define
\begin{equation}\label{eq:mfn-coords}
\begin{aligned}
X_0&=x,&X_1&=y,&X_2&=z,&X_7&=t,\\
X_3&=t(x+y+z),&X_4&=t(x+y),&X_5&=t^2(x+z),&X_6&=t(y+z),\\
X_8&=y+\frac{x+y}{y+z},&&
X_9=y+t(x+y)(y+z).&&
\end{aligned}
\end{equation}
These ten functions have algebraic matroid exactly \eqref{eq:mfn-ch}.  This assertion is verified by an exact computation.  Set $a=x+y$ and $b=y+z$.  In characteristic $2$ the inverse substitution is $x=a+y$ and $z=b+y$, so $(a,b,y,t)$ is a transcendence basis of $K$.  The eleven listed four-sets satisfy explicit nonzero polynomial identities in their four indexed outputs; among the other $199$ four-sets, $190$ have a nonzero $4\times4$ Jacobian determinant with respect to $(a,b,y,t)$.  The remaining nine,
\[
0125,\ 0158,\ 0235,\ 0245,\ 0256,\ 0258,\ 0259,\ 1258,\ 1789,
\]
have zero ordinary Jacobian because of inseparability in characteristic $2$, but exact saturated Gr\"obner elimination gives zero elimination ideal in each case.  Equivalently, the source field is algebraic over the field generated by the corresponding four outputs.  Finally, every three-set extends to one of these verified bases.  The companion certificate contains the exact Jacobian numerators, the nine Gr\"obner certificates, and a basis extension for every three-set.  Hence \eqref{eq:mfn-coords} proves $M^{\mathrm{FN}}_{10}\in\Alg$.  Every set of size at most three is independent, so the matroid is simple.  For connectedness, the first five circuit-hyperplanes in \eqref{eq:mfn-ch} all contain $0$ and their union is $E$; hence every element lies with $0$ in a circuit.

Two minors give the obstruction to folded linearity:
\begin{equation}\label{eq:mfn-minors}
M^{\mathrm{FN}}_{10}/7\setminus\{8,9\}\cong F_7,
\qquad
M^{\mathrm{FN}}_{10}/1\setminus\{3,5\}\cong F_7^-.
\end{equation}
Mat\'u\v{s} proved that every partition representation of $F_7$ has degree a power of two, whereas every partition representation of $F_7^-$ has odd degree~\cite[Proposition~4.1]{Matus1999}.  A partition representation passes to every minor without changing its degree: deletion discards a partition, while contraction restricts the remaining partitions to a block of the contracted partition.  Since a nontrivial partition representation has degree at least two, no common degree can represent both minors.  Hence $M^{\mathrm{FN}}_{10}$ is not partition representable.  Every folded-linear matroid is partition representable (a finite-field $t$-linear representation gives the standard partition representation of degree $q^t$; finite representations may be specialized to a finite field), so $M^{\mathrm{FN}}_{10}\notin\FLin$.
\end{proof}

By Mat\'u\v{s}'s theorem that algebraic matroids are almost entropic~\cite{Matus2024}, the same example is almost entropic.  Thus \cref{prop:mfn10} provides a compact explicit point of $\Alg\setminus\FLin\subseteq\FAlg$.

For the eight-element frontier, we use the notation $P_1,P'_2,P''_2,P_3,L'_8$ of~\cite{BamiloshinEtAl2021}.  A companion paper proves that $P_1$ (denoted $P_{8,1}$ there) is not almost entropic~\cite{KhazaeiP81}.  Since $\Alg\subseteq\FAlg\subseteq\AlmostEnt$, this matroid is not algebraic.  The unresolved eight-element candidates are therefore $P'_2$ and $P''_2$ (denoted $P'_{8,2}$ and $P''_{8,2}$ in the companion paper): both are known not to be folded linear, while their algebraicity remains open~\cite{BamiloshinFarrasPadro2025}.  They are distinct from the nine-element matroids with Bollen identifiers 156374 and 156390.  Hence the certified interval remains $8\le n_{\min}(\Alg\setminus\FLin)\le10$.

\subsection{Almost-entropic and characteristic properties}\label{subsec:ae-characteristics}

\begin{definition}[Almost entropic]\label{def:ae}
For jointly distributed discrete random variables $(X_e)_{e\in E}$, the function
\[
h(A)=H(X_A),
\qquad X_A=(X_e)_{e\in A},
\]
is an \emph{entropic polymatroid}.  The closed conic hull of entropic polymatroids is the \emph{almost-entropic cone}.  A matroid $M$ is \emph{almost entropic} if its rank function lies in that cone.  We write $\AlmostEnt$ for this class.
\end{definition}

\begin{theorem}[Upper inclusion]\label{thm:falg-ae}
Every folded-algebraic matroid is almost entropic:
\[
\FAlg\subseteq\AlmostEnt.
\]
\end{theorem}
\begin{proof}
Let $N$ be a $t$-block algebraic lift of $M$.  Mat\'u\v{s} proved that every algebraic matroid is almost entropic~\cite{Matus2024}.  Approximate $r_N$ by entropic polymatroids on the scalar ground set and group the $t$ random variables in each block into one vector-valued variable.  On unions of whole blocks the limiting entropy is $t\,r_M(A)$.  Since the almost-entropic cone is conic, division by $t$ yields $r_M\in\AlmostEnt$.
\end{proof}

For characteristic arguments, let $D_c(M)$ be the set of folding degrees $t$ for which $M$ has a $t$-fold algebraic representation in characteristic $c$, and let $D_c^{\FLin}(M)$ be the corresponding set for $t$-fold linear representations.  Write
\[
\chi_{\FAlg}(M)=\{c:D_c(M)\ne\varnothing\}.
\]
Here $c$ ranges over $0$ and the prime characteristics.

\begin{proposition}[Characteristic-zero collapse]\label{prop:charzero}
For every matroid $M$ and every $t\ge1$,
\[
M\text{ is $t$-fold algebraic in characteristic }0
\iff
M\text{ is $t$-fold linear in characteristic }0.
\]
Equivalently, the degree spectra agree in characteristic zero:
\[
D_0(M)=D^{\FLin}_0(M).
\]
\end{proposition}
For $t=1$, this is Ingleton's characteristic-zero theorem~\cite{Ingleton1971}.  The proposition shows that the differential argument preserves the folding degree.
\begin{proof}
A $t$-fold linear representation gives a $t$-fold algebraic representation by the linear-form construction following \cref{lem:blocklift}.  Conversely, let $(X_e)_{e\in E}$ be a $t$-fold algebraic representation over a characteristic-zero field $K$, put $X_A=\bigcup_{e\in A}X_e$, and write $F=K(X_E)$.  In the $F$-vector space $\Omega_{F/K}$ of K\"ahler differentials, let $V_e$ be the span of the differentials of the coordinates in $X_e$.  In characteristic zero, algebraic independence is equivalent to linear independence of differentials, so for every $A\subseteq E$,
\[
\dim_F\sum_{e\in A}V_e=\trdeg_K K(X_A)=t\,r_M(A).
\]
Thus the subspaces $V_e$ form a $t$-fold linear representation.
\end{proof}

\begin{proposition}[Direct sums]\label{prop:directsum}
For matroids $M_1,M_2$ and every characteristic $c$,
\[
D_c(M_1\oplus M_2)=D_c(M_1)\cap D_c(M_2),
\]
and therefore
\[
\chi_{\FAlg}(M_1\oplus M_2)=
\chi_{\FAlg}(M_1)\cap\chi_{\FAlg}(M_2).
\]
\end{proposition}
\begin{proof}
A folded-algebraic representation of the direct sum restricts to each summand, so only the converse needs proof.

If a common degree $t$ lies in both spectra and the characteristic is positive, apply the Piff--Lindstr\"om reduction to the two scalar block lifts~\cite[Theorem~2.3]{BenEfraim2016}.  Both labeled algebraic matroids may then be represented over the same prime field.  Take isomorphic copies of their ambient fields that are algebraically independent over that prime field.  Since transcendence degrees add across algebraically independent composita, the union of the two block systems is a $t$-fold representation of $M_1\oplus M_2$.  In characteristic zero, no reduction is needed: place algebraically independent copies of the two block systems inside one sufficiently large algebraically closed field, again making transcendence degrees additive.

For the characteristic-set consequence one may start with degrees $s\in D_c(M_1)$ and $t\in D_c(M_2)$ that are not equal.  Repeating the first representation $t$ times and the second $s$ times gives a common degree $st$, and the previous paragraph applies.
\end{proof}

Characteristic sets for linear and algebraic matroids were studied in~\cite{Kahn1982,Lindstrom1985,Gordon1988,CartwrightVarghese2024}.  The rank-three family $M(p)$ used below is the classical family of Gordon~\cite{Gordon1988}, related to Lindstr\"om's earlier work on algebraic characteristic sets~\cite{Lindstrom1985} and later studied by Fl\'orez~\cite{Florez2006}.  Fl\'orez gives the convenient line description and records $M(2)=F_7$.  Our theorem strengthens the classical characteristic statement by proving that the same obstruction persists under arbitrary finite algebraic folding.  Group-configuration methods enter through the corrected abelian theorem of Bays--Hils--Moosa and the formulation for algebraically closed fields due to Bays--Breuillard~\cite{BaysHilsMoosa2017,BaysBreuillard2021}.  For the surrounding representation landscape and its standard class diagrams we refer especially to Bamiloshin--Ben-Efraim--Farr\`as--Padr\'o~\cite{BamiloshinEtAl2021}; for modern almost-multilinear and almost-entropic context, see K\"uhne--Yashfe~\cite{KuhneYashfeEntropic}.  For comparing the groups recovered from different cells we use the quasi-isomorphism theorem of Boege--Yashfe~\cite{BoegeYashfe2026}, combined below with a finite-quotient construction that places all recovered cells in one common group.

\section{The classical matroids \texorpdfstring{$M(p)$}{M(p)}}\label{sec:Mp}

Fix a prime $p$.  Let
\[
E_p=\{A,X,Y\}\cup\{S_i,R_i:i\in\mathbf F_p\},
\]
and represent these points over $\mathbf F_p$ by
\begin{equation}\label{eq:Mp-matrix}
A=(1,0,0),\quad X=(0,0,1),\quad Y=(1,0,1),\qquad
S_i=(i,1,0),\quad R_i=(i,1,1).
\end{equation}
Call the resulting rank-three matroid $C_p$.

\begin{proposition}\label{prop:lines}
The nontrivial rank-two flats of $C_p$ are exactly
\[
L_S=\{A,S_i:i\in\mathbf F_p\},\qquad
L_R=\{A,R_i:i\in\mathbf F_p\},
\]
\[
L_\infty=\{A,X,Y\},\qquad
V_i=\{X,S_i,R_i\},\qquad
D_i=\{Y,S_i,R_{i+1}\}.
\]
Hence $C_p$ is simple, rank three, and isomorphic to the classical $M(p)$ under
\[
d\leftrightarrow A,\quad a_i\leftrightarrow S_i,\quad b_i\leftrightarrow R_i,
\quad c_0\leftrightarrow X,\quad c_1\leftrightarrow Y.
\]
In particular $M(2)=F_7$.
\end{proposition}
\begin{proof}
A projective line has equation $\alpha u+\beta v+\gamma w=0$.  Substitution into \eqref{eq:Mp-matrix} gives
\[
S_i\in\ell\iff \alpha i+\beta=0,
\qquad
R_i\in\ell\iff \alpha i+\beta+\gamma=0,
\]
while $A\in\ell$ iff $\alpha=0$, $X\in\ell$ iff $\gamma=0$, and $Y\in\ell$ iff $\alpha+\gamma=0$.

If $\alpha=0$, then the line contains $A$.  If also $\beta=0$, it is $L_S$; if instead $\beta+\gamma=0$, it is $L_R$; and if $\gamma=0$, it is $L_\infty$.  If $\alpha\neq0$, then among the $S_i$ there is exactly one solution of $\alpha i+\beta=0$, and among the $R_i$ exactly one solution of $\alpha i+\beta+\gamma=0$.  If $\gamma=0$ the line is some $V_i$ through $X$; if $\alpha+\gamma=0$ it is some $D_i$ through $Y$; and otherwise no third distinguished point lies on the line.  This yields exactly the displayed list, agreeing with the standard description of $M(p)$ recorded by Fl\'orez~\cite{Florez2006}.
\end{proof}

For each $i\in\mathbf F_p$ define the seven-point cell
\begin{equation}\label{eq:cell}
Q_i=(A,S_i,S_{i+1},R_i,X,Y,R_{i+1}),
\end{equation}
with roles $(A,B,C,W,X,Y,Z)$.  Translating the first projective coordinate by $i$ normalizes the cell to
\[
A=(1,0,0),\ B=(0,1,0),\ C=(1,1,0),\ W=(0,1,1),
\]
\[
X=(0,0,1),\ Y=(1,0,1),\ Z=(1,1,1).
\]
The six dependent triples are
\begin{equation}\label{eq:sixlines}
ABC,\quad AXY,\quad BYZ,\quad CXZ,\quad AWZ,\quad BWX.
\end{equation}
Every other triple is independent except that $(W,C,Y)$ is also dependent when $p=2$.  Indeed, the slope-one line through $Y$ and $S_{i+1}$ is $D_{i+1}=\{Y,S_{i+1},R_{i+2}\}$, so $\{R_i,S_{i+1},Y\}$ is a line if and only if $R_i=R_{i+2}$, equivalently if and only if $p=2$.  Thus $Q_i\cong F_7^-$ for odd $p$ and $Q_i\cong F_7$ for $p=2$.

No single cell can force the characteristic: the non-Fano matroid is algebraic over every field.  This follows from the classical monomial construction of Lindstr\"om; the same argument is the one-dimensional-group viewpoint used by Bollen--Cartwright--Draisma~\cite{Lindstrom1985,BollenCartwrightDraisma2022}.

\begin{lemma}[Monomial realizations from rational vectors]\label{lem:monomial}
If a matroid has a vector representation over $\mathbb Q$ (equivalently over $\mathbb Z$), then it has an algebraic representation over every field.
\end{lemma}
\begin{proof}
Let $v_e=(n_{e1},\dots,n_{em})\in\mathbb Z^m$ be a vector representation and let $t_1,\dots,t_m$ be algebraically independent over the base field.  Represent $e$ by the monomial
\[
x_e=t_1^{n_{e1}}\cdots t_m^{n_{em}}.
\]
For every set $A$, the field generated by $(x_e)_{e\in A}$ has transcendence degree equal to the $\mathbb Q$-rank of the exponent vectors $v_e$ with $e\in A$, hence equal to the matroid rank.  This computation is independent of the characteristic.
\end{proof}

In particular, the non-Fano matroid has the monomial realization
\[
a,\ b,\ ab,\ u,\ au,\ bu,\ abu
\]
over every field.  The characteristic obstruction proved below therefore comes from the cyclic gluing of the cells, not from an individual cell.

The following construction locates $\FAlg$ relative to its two principal subclasses and gives the two improved upper bounds in \cref{tab:witness-sizes}.

\subsection{A folded-algebraic-only witness}\label{subsec:m33}

If $L\subseteq E(M)$, the \emph{principal extension} $M+_Lz$ adds $z$ freely to the flat $\cl_M(L)$: for $S\subseteq E(M)$, the new point is in $\cl(S)$ exactly when $L\subseteq\cl_M(S)$.

\begin{lemma}[Principal extensions at fixed field and degree]\label{lem:principal-fixed}
Fix a field $K$ and a positive integer $t$.  The class of matroids having a $t$-fold algebraic representation over $K$ is closed under principal extensions, deletions, and contractions.
\end{lemma}
\begin{proof}
Let $M$ have an algebraic block lift $N$ over $K$, with blocks $X_e$ of size $t$, and write $X_A=\bigcup_{e\in A}X_e$.  Thus
\[
r_N(X_A)=t\,r_M(A)
\qquad(A\subseteq E(M)).
\]
Let $M+_L z$ be the principal extension obtained by adding $z$ freely to the flat spanned by $L$, and put $G=\cl_N(X_L)$.  Ordinary algebraic matroids over a fixed field are closed under principal extensions~\cite[Lemma~13]{Matus2018}.  In the scalar lift, successively add $z_1,\ldots,z_t$ freely to the closure of $G$ in the resulting matroid at each stage.  The final scalar matroid remains algebraic over $K$.

For a block union $X_A$, its rank deficit to $G$ is
\[
d_A=r_N(X_A\cup X_L)-r_N(X_A)
=t\bigl(r_M(A\cup L)-r_M(A)\bigr).
\]
If $L\subseteq\cl_M(A)$, then $d_A=0$ and every $z_j$ lies in the closure of $X_A$.  Otherwise $d_A$ is a positive multiple of $t$, and hence $d_A\ge t$.  Before $z_j$ is added, the preceding $j-1<t$ new points can reduce this deficit by at most $j-1$.  Thus $z_j$ raises the rank by one.  Consequently, with $X_z=\{z_1,\ldots,z_t\}$,
\[
r(X_A\cup X_z)=
\begin{cases}
t\,r_M(A),&L\subseteq\cl_M(A),\\
t\bigl(r_M(A)+1\bigr),&L\nsubseteq\cl_M(A).
\end{cases}
\]
This is a $t$-fold algebraic representation of $M+_L z$ over the same field.

Deletion discards a block.  If $e$ is a nonloop, then for every surviving block union,
\[
r_{N/X_e}(X_A)
=r_N(X_A\cup X_e)-r_N(X_e)
=t\,r_{M/e}(A).
\]
Ordinary algebraic matroids are minor closed over the same field; contraction of a loop is the same as deletion.  This proves the lemma.
\end{proof}

\begin{proposition}[A $33$-element folded-algebraic-only witness]\label{prop:m33}
There is a rank-four matroid $A_{24}^{(5)}$ on $24$ elements and a rank-six matroid $M_{33}$ on $33$ elements such that
\[
A_{24}^{(5)}\in\FLin\setminus\Alg,
\qquad
M_{33}\in\FAlg\setminus(\Alg\cup\FLin).
\]
Both positive representations use characteristic $5$; $A_{24}^{(5)}$ is $2$-linear and $M_{33}$ is $2$-fold algebraic.
\end{proposition}
\begin{proof}
We give the construction in four steps.  The finite rank assertions used below are checked by the exact verifier in the computational companion.  The characteristic obstructions are established by the mathematical arguments in the proof.

\emph{A characteristic-$5$ core.}
We define the $15$-element core directly.  Put $W=\mathbb F_5^2$ and let $I$ denote its identity matrix.  In $W^3$, let $p_1,p_2,p_3$ be the images of $v\mapsto(v,0,0)$, $v\mapsto(0,v,0)$, and $v\mapsto(0,0,v)$, respectively, and let $O$ be the image of $v\mapsto(v,v,v)$.  Set $\rho(-e)=-I$ and
\[
\rho(i)=\begin{pmatrix}0&1\\4&0\end{pmatrix},\qquad
\rho(j)=\begin{pmatrix}0&2\\2&0\end{pmatrix},\qquad
\rho(k)=\begin{pmatrix}2&0\\0&3\end{pmatrix}.
\]
For $g\in\{-e,i,j,k\}$, write $g^1,g^2,g^3$ for the images of
\[
v\longmapsto(-v,\rho(g)v,0),\qquad
v\longmapsto(0,-v,\rho(g)v),\qquad
v\longmapsto(\rho(g)v,0,-v),
\]
respectively.  Let $C_{15}$ be the $2$-linear matroid on
\[
\{p_1,p_2,p_3,O,-e^1,-e^2,-e^3,
  i^1,i^2,i^3,j^1,j^2,j^3,k^1,k^2\}.
\]
This is the indicated restriction of Ben-Efraim's $Q^{\mathrm{NF}}_3(\mathsf Q_8)$~\cite{BenEfraim2016}.  The computational companion also constructs an explicit width-two model of the same core over $\mathbb F_7$.  An exhaustive comparison of all $2^{15}$ subsets shows that every block-span dimension is even and that the $\mathbb F_5$ and $\mathbb F_7$ models have the same span dimension on every subset.

We claim that $C_{15}$ is not algebraic in characteristic $5$.  Suppose otherwise.  By the Piff--Lindstr\"om reduction, we may work over the prime field~\cite[Theorem~2.3]{BenEfraim2016}.  The Frobenius/derivation preparation in~\cite[Lemma~4.3 and Claims~4.4--4.6]{BenEfraim2016} localizes to this restriction: Claim~4.4 uses only the seven-point non-Fano seed, and Claims~4.5--4.6 adjust one further edge element at a time.  Thus, after individual Frobenius replacements, all selected elements are separable over a basis and have nonzero gradients.  Original circuits remain dependent in the derivative shadow.  Claim~4.4 also makes the seed shadow simple of rank three: every pair of seed gradients is independent, and every triple among $p_1,p_2,p_3,O$ is independent.  The six non-Fano lines survive.  The only further line compatible with these incidences would complete the Fano plane, which is not linearly representable in characteristic $5$.  Hence no degeneration occurs, and the seed gradients have the standard non-Fano normalization
\[
\begin{aligned}
p_1&=(1,0,0),&p_2&=(0,1,0),&p_3&=(0,0,1),\\
-e^1&=(1,1,0),&-e^2&=(0,1,1),&-e^3&=(1,0,1),&O&=(1,1,1).
\end{aligned}
\]
Write three nonzero edge vectors as
\[
U_1=(x_1,y_1,0),\quad U_2=(0,y_2,z_2),\quad U_3=(x_3,0,z_3).
\]
Dependence of $(-e^1,U_2,U_3)$, $(U_1,-e^2,U_3)$ and $(U_1,U_2,-e^3)$ gives
\[
x_3z_2+y_2z_3=x_1z_3+x_3y_1=x_1y_2+y_1z_2=0.
\]
None of the six displayed coordinates can vanish.  For example, if $x_1=0$, then $x_1y_2+y_1z_2=0$ gives $z_2=0$, and $x_1z_3+x_3y_1=0$ gives $x_3=0$; the first equation then gives $y_2z_3=0$, contradicting the nonzeroness of $U_2$ and $U_3$.  The other five cases are symmetric.  Projectively normalize to
$U_1=(1,a,0)$, $U_2=(0,1,b)$ and $U_3=(c,0,1)$.  Then $ab=ac=bc=-1$, so $a=b=c$ and $a^2=-1$.  Applying this to the $i$- and $j$-triangles gives parameters $\alpha,\beta$ with $\alpha^2=\beta^2=-1$.

Use the normalizations
\[
i^1=(1,\alpha,0),\quad i^2=(0,1,\alpha),\quad i^3=(\alpha,0,1),
\]
\[
j^1=(1,\beta,0),\quad j^2=(0,1,\beta),\quad j^3=(\beta,0,1).
\]
The edge vectors $k^1,k^2$ have both displayed coordinates nonzero by the same vanishing argument, so write
\[
k^1=(1,K_1,0),\qquad k^2=(0,1,K_2).
\]
Taking determinants for the remaining circuits $(j^1,k^2,i^3)$, $(k^1,i^2,j^3)$ and $(k^1,k^2,-e^3)$ gives
\[
\alpha\beta K_2=-1,\qquad K_1\alpha\beta=-1,\qquad K_1K_2=-1.
\]
The first two equations imply $K_1=K_2=-1/(\alpha\beta)$, hence $K_1K_2=1/(\alpha^2\beta^2)=1$, contradicting the third equation in characteristic $5$.  Thus
\begin{equation}\label{eq:c15-nonalg}
C_{15}\notin\Alg\quad\text{in characteristic }5.
\end{equation}

\emph{The $24$-element component.}
Take the preceding $C_{15}$ representation and the scalar width-$2$ lift of the characteristic-$5$ Reid geometry $R_5=M(5)$.  Identify the packets $p_1\leftrightarrow A$ and $p_2\leftrightarrow X$, using on the second packet the change of basis
\[
T=\begin{pmatrix}1&2\\1&3\end{pmatrix},
\]
and then delete the two common ground elements.  The matrix $T$ specifies the isomorphism between the two-dimensional packets being identified.  Although a change of basis does not alter either packet separately, the chosen inter-packet identification affects the quotient arrangement; this choice makes the two recovery minors below survive.  The resulting packet arrangement $A_{24}^{(5)}$ lies in an $8$-dimensional vector space over $\mathbb F_5$, and the superscript records the characteristic.  The exact verifier enumerates the $783$ distinct packet spans reached by all $2^{24}$ subsets.  Equal spans are merged because every tested rank depends only on the resulting span; the enumeration is therefore exhaustive.  Every scalar rank is even and the total scalar rank is $8$.  Hence $A_{24}^{(5)}$ is $2$-linear of matroid rank $4$.

Write $C{:}x$ and $R{:}y$ for elements inherited from the $C_{15}$ and $R_5$ sides, respectively.  The same calculation recovers both factors as minors.  For $C_{15}$, contract $R:S_0$; the elements $R:S_1$ and $R:R_0$ replace the two deleted common points $p_1,p_2$.  For $R_5$, contract $C:O$; the elements $C:-e^2$ and $C:-e^3$ replace $A,X$.  If $A_{24}^{(5)}$ were algebraic, its $R_5$ minor would force characteristic $5$ by Gordon's Reid-geometry theorem~\cite{Gordon1988}, whereas its $C_{15}$ minor would contradict \eqref{eq:c15-nonalg}.  Therefore
\[
A_{24}^{(5)}\in\FLin\setminus\Alg.
\]

\emph{The generalized non-Fano component.}
Let $e_1,\ldots,e_6$ be the standard basis of $\mathbb Q^6$ and put $\mathbf 1=e_1+\cdots+e_6$.  Define $B_5$ to be the rank-six rational matroid represented by
\[
a_i=e_i,\qquad b_i=\mathbf 1-e_i\quad(1\le i\le6),
\qquad c=\mathbf 1.
\]
This is the matroid $NnF(5)$ of Sharififar--Sadeghi--Aboutorab~\cite[Definition~21]{SharififarEtAl2024}.  In particular, $|E(B_5)|=13$, $r(B_5)=6$, and $(a_1,a_2,a_3,a_4)$ is independent.  Over characteristic $5$ it has the monomial realization
\[
a_i=x_i,\qquad b_i=\prod_{j\ne i}x_j,
\qquad c=\prod_{j=1}^6x_j,
\]
with the $x_i$ algebraically independent.  By \cref{lem:monomial}, rational exponent rank gives its matroid rank, so $B_5$ is algebraic.

Sharififar--Sadeghi--Aboutorab prove that $B_5$ has no $t$-linear representation over a finite field of characteristic $5$, for any $t$~\cite[Proposition~7]{SharififarEtAl2024}.  This implies the same statement over every field of characteristic $5$.  Indeed, a block representation over such a field uses finitely many matrix entries.  Adjoin those entries to $\mathbb F_5$ and invert one nonzero minor witnessing each required rank.  The resulting finitely generated $\mathbb F_5$-algebra has a maximal-ideal quotient that is a finite field.  All minors required to vanish still vanish, while the chosen rank minors remain nonzero.  Applying the singleton rank conditions ensures that every packet retains dimension $t$ after reduction.  We therefore obtain a block representation of the same degree over a finite field, contradicting the cited proposition.  Hence
\begin{equation}\label{eq:nnf5-noflin}
B_5\text{ has no $t$-linear representation for any }t\ge1.
\end{equation}

\emph{Four-point independent-set welding.}
In $A_{24}^{(5)}$ and $B_5$, respectively, take the ordered sets
\[
I_A=(C:p_3,C:{-}e^1,C:{-}e^2,R:S_0),
\qquad
I_B=(a_1,a_2,a_3,a_4).
\]
The exact rank calculation gives $r_A(I_A)=r_B(I_B)=4$.  Thus $I_A$ is a basis of the rank-four matroid $A_{24}^{(5)}$, whereas $I_B$ is only an independent set in the rank-six matroid $B_5$.  To identify a pair of independent points $x,y$, add a point freely to $\cl(x,y)$, contract the new point, delete $y$, and retain $x$ as the common label.  Starting from the direct sum, perform this operation on the four pairs successively, and retain the $A$-side labels.  Denote the resulting matroid by $M_{33}$.  In the deposited rank oracle, elements from $A_{24}^{(5)}$ have labels of the form \texttt{A:...}, the nine unshared elements from $B_5$ have labels of the form \texttt{B:...}, and each identified element appears once under its retained \texttt{A:...} label.  For $J\subseteq[4]$, let $I_A(J)$ and $I_B(J)$ be the corresponding subtuples.  If $S_A,S_B$ are the preimages of $S$ on the two factor ground sets, induction over the four identifications gives
\begin{equation}\label{eq:m33-rank}
r_{M_{33}}(S)=
\min_{J\subseteq[4]}
\bigl(r_A(S_A\cup I_A(J))+r_B(S_B\cup I_B(J))-|J|\bigr).
\end{equation}
Both factors survive as natural restrictions.  For a set contained in the $A$-factor, independence of $I_B$ contributes $|J|$ to every term of \eqref{eq:m33-rank}; thus all terms are at least its $A$-rank, and $J=\varnothing$ gives equality.  For a set $T$ in the $B$-factor, let $I$ index the selected elements contained in $T$.  In the term indexed by $J$, independence of $I_A$ contributes $|I\cup J|$, while removing the elements indexed by $I\setminus J$ lowers $r_B(T)$ by at most $|I\setminus J|$.  Hence every term is at least
\[
|I\cup J|+r_B(T)-|I\setminus J|-|J|=r_B(T),
\]
and equality is attained at $J=I$.  The $B_5$ restriction is spanning, and hence
\[
|E(M_{33})|=24+13-4=33,
\qquad r(M_{33})=6.
\]
The four elements of $I_B$ are independent but do not form a basis of the rank-six matroid $B_5$; thus the construction is a four-point independent-set welding, not a welding of two bases.

The matroid $A_{24}^{(5)}$ is $2$-linear, and therefore $2$-fold algebraic, over $\mathbb F_5$.  Taking two algebraically independent copies of the monomial realization makes $B_5$ $2$-fold algebraic over the same field.  Choose the two scalar lifts algebraically independent over $\mathbb F_5$ before forming their direct sum.  \Cref{lem:principal-fixed} applies at each principal extension, contraction, and deletion, and gives
\[
M_{33}\in\FAlg
\]
over $\mathbb F_5$ with folding degree $2$.

Since $A_{24}^{(5)}$ is a restriction, $M_{33}$ cannot be algebraic.  If $M_{33}$ were $t$-linear over a field $K$, its $A_{24}^{(5)}$ restriction, and then its $R_5$ minor, would force $\operatorname{char}K=5$ by the arbitrary-width Reid theorem~\cite[Theorem~3.1]{BenEfraim2016}.  Its $B_5$ restriction would then contradict \eqref{eq:nnf5-noflin}.  Hence $M_{33}\notin\FLin$.

The companion verifier evaluates \eqref{eq:m33-rank} on all $783$ reachable packet-span states of $A_{24}^{(5)}$ and all $2^{13}=8192$ subsets of $B_5$.  It checks both factor restrictions, the displayed size and rank, and simplicity on all $528$ pairs.  This completes the proof.
\end{proof}

\subsection{Computational certificates}\label{subsec:certificates}

The exact certificates for $M^{\mathrm{FN}}_{10}$, $A_{24}^{(5)}$, $M_{33}$, and $C_{2,3}$ are contained in the computational companion archived at Zenodo, version~1.1.0, DOI~\href{https://doi.org/10.5281/zenodo.22178786}{10.5281/zenodo.22178786}~\cite{KhazaeiCompanion2026}.  The $M^{\mathrm{FN}}_{10}$ verifier uses Python~3 and SymPy~1.14.0 for exact arithmetic over $\mathbb F_2$ and exact Gr\"obner elimination.  The other verifiers use Python~3 standard-library rational and finite-field arithmetic.  The archive contains the input matrices and exponent vectors, rank oracles for $M_{33}$ and $C_{2,3}$, canonical JSON certificates, run scripts, and a SHA-256 manifest of every distributed file.  The proofs of \cref{prop:mfn10,prop:m33,thm:C23,cor:dealer} identify the finite assertions checked by each certificate.

\section{Folding-stable characteristic rigidity}\label{sec:proof}

We prove \cref{thm:intro-singleton}.  Fix a prime $p$ and suppose $M(p)$ has a $t$-fold algebraic representation over a field $K$.  Passing to an algebraic closure and using \cref{lem:blocklift}, we may work in one sufficiently saturated algebraically closed extension.

\subsection{Full embeddings and the group input}

Write $\acl_F(S)$ for field-theoretic algebraic closure over an algebraically closed base $F$, and write $x\simF y$ when $\acl_F(x)=\acl_F(y)$.  Tuples satisfying this relation are \emph{interalgebraic over $F$}.

We use standard notation from the model theory of algebraically closed fields.  The abbreviation ACF denotes the first-order theory of algebraically closed fields; $\mathbb U$ is a sufficiently saturated algebraically closed field, and $F_0\preceq\mathbb U$ means that $F_0$ is an elementary subfield.  The relation $F\ind_{F_0}D$ means algebraic independence over $F_0$.  An $F$-definable set is one given by first-order field conditions with parameters in $F$; its dimension is transcendence degree over $F$.  A generic tuple has dimension equal to that of the definable set.  A definable group is \emph{connected} if it has no proper definable subgroup of finite index.

\begin{lemma}[Full-embedding dictionary]\label{lem:full}
For a simple matroid with a $t$-fold algebraic representation over an algebraically closed base $F$, every point has dimension $t$, every pair of distinct points has dimension $2t$, a collinear triple has dimension $2t$ and each point lies in the algebraic closure of the other two, while a noncollinear triple has dimension $3t$ and is independent.
\end{lemma}
\begin{proof}
These are exactly the rank identities \eqref{eq:falg}.  For a collinear triple $x,y,z$, for example,
\[
\trdeg_F(x,y,z)=2t=\trdeg_F(x,y),
\]
so $z\in\acl_F(x,y)$, and similarly cyclically.
\end{proof}

Choose a field $\mathbb U$ as above that contains the representation over an algebraically closed $F_0$.  Since ACF is model complete, regard $F_0\preceq\mathbb U$.  For the finitely many cells, the full group-configuration theorem may enlarge the base by parameter sets $B_0,\ldots,B_{p-1}$ independent from the corresponding configurations.  By extension and homogeneity, choose these parameter sets jointly independent from the finite tuple $D$ of all embedded gadget coordinates over $F_0$, and then choose a sufficiently saturated algebraically closed elementary submodel $F\supseteq F_0B_0\cdots B_{p-1}$ with
\[
F\ind_{F_0}D.
\]
Then every subtuple of $D$ has the same transcendence degree over $F$ as over $F_0$, so all rank, closure, and genericity data are preserved.  All auxiliary parameter sets are contained in this single base, and we perform every cell recovery over $F$.

The proof uses two group-configuration theorems.  The statements needed here are characteristic-free: Bays--Breuillard explicitly work with arbitrary algebraically closed fields in their group-configuration paragraph~\cite[Section~3.0.1]{BaysBreuillard2021}.  Their characteristic-zero restrictions are confined to the separate incidence estimate and converse construction in their Lemma~2.15 and Proposition~7.10.  We use the full group-configuration theorem in its group-plus-homogeneous-space form, stated explicitly in the corrected online version of~\cite[Theorem~6.1]{BaysGST2018}; that exposition follows~\cite[Theorem~5.4.5 and Remark~5.4.10]{Pillay1996}.  The corrected seven-point theorem independently recovers a connected commutative algebraic group~\cite[Theorem~3.3]{BaysBreuillard2021}.  The latter statement explicitly permits the exceptional triple $(W,C,Y)$ that occurs when $p=2$; its commutativity input is due to Bays--Hils--Moosa~\cite[Theorem~C.1]{BaysHilsMoosa2017}.

\begin{theorem}[Seven-point recovery]\label{thm:recovery}
Suppose seven $t$-dimensional points $A,B,C,W,X,Y,Z$ satisfy the six lines \eqref{eq:sixlines}, every other triple being independent except possibly $(W,C,Y)$.  Then there is a connected commutative $F$-definable group $G$ of dimension $t$ and representatives
\[
a,b,c,u,y,z\in G
\]
interalgebraic with $A,B,C,X,Y,Z$, respectively, such that $a,b,u$ are independent generics and
\[
c=a+b,\qquad y=a+u,\qquad z=a+b+u
\]
are literal group equalities.
\end{theorem}

Apply the full group-configuration theorem to the six-point configuration inside the given seven-point cell.  It first yields a connected faithful $\bigvee$-definable homogeneous space $(\widetilde G,S)$ over the enlarged base and a group configuration interalgebraic with $(A,B,C,X,Y,Z)$, with the corresponding group and action identities.  In ACF, the algebraic reformulation immediately following the cited theorem replaces this generic action by a birational action of a connected algebraic group on an irreducible variety, through generically finite correspondences.  Thus, after interalgebraic replacement, the generic group and action coordinates may be taken in connected $F$-definable objects.

Separately, the seven-point theorem recovers a connected commutative $F$-definable algebraic group $G'$ and a group triple interalgebraic with $(A,B,C)$.  Apply the comparison theorem stated below as \cref{thm:compare} to this triple and the group triple in $\widetilde G$.  After quotienting the finite kernels and target indeterminacy as in \eqref{eq:qiso-quotient}, and quotienting the homogeneous space by the corresponding finite action, the generic coordinates remain interalgebraic and the action identities remain literal.  The resulting connected commutative group acts transitively and faithfully.  For a commutative transitive action all point stabilizers coincide with the action kernel, so faithfulness makes the action regular; the homogeneous space is therefore principal.  It is $F$-definable and nonempty, hence has an $F$-point because $F$ is algebraically closed.  Choosing that point as origin identifies the homogeneous space with the group and converts the action identities into
\[
c=a+b,\qquad y=a+u,\qquad z=a+b+u.
\]
Finally, $a$ is generic and interalgebraic with the $t$-dimensional point $A$, so the resulting group has dimension $t$.

The missing seventh vertex is forced as well.

\begin{lemma}[Cube completion]\label{lem:cube}
With the notation of \cref{thm:recovery},
\[
W\simF b+u.
\]
\end{lemma}
\begin{proof}
Let $w$ be any representative of the embedded vertex $W$, and put
\[
P=\acl_F(b,u),\qquad Q=\acl_F(a,a+b+u).
\]
Each has dimension $2t$ and their join has dimension $3t$, so $\dim(P\cap Q)\le t$.  The lines $BWX$ and $AWZ$ put $w$ in both closures, and $w$ has dimension $t$; hence $P\cap Q=\acl_F(w)$.  The element $b+u$ also has dimension $t$ and belongs to both closures, so $w\simF b+u$.
\end{proof}

Second, we use the comparison theorem of Boege--Yashfe~\cite[Theorem~3.14]{BoegeYashfe2026}.  A \emph{group triple} is a triple $(x,y,x+y)$ of generic elements with the expected pairwise dimensions.  A \emph{quasi-epimorphism} from a connected $F$-definable group $G$ to a connected $F$-definable group $H$ is an $F$-definable subgroup $\Phi\le G\times H$ such that both coordinate projections are surjective and the kernel of the projection $\Phi\to G$ is finite; it is a \emph{quasi-isomorphism} if the kernel of the projection $\Phi\to H$ is finite as well.

\begin{theorem}[Comparison of group triples]\label{thm:compare}
Let $(x,y,z)$ and $(x',y',z')$ be group triples in connected $F$-definable groups $G$ and $G'$.  Suppose
\[
x'\in\acl_F(x),\qquad y'\in\acl_F(y),\qquad z'\in\acl_F(z).
\]
Then there are $\alpha,\beta\in G'(F)$ and an $F$-definable quasi-epimorphism $\Phi\le G\times G'$ such that, in additive notation,
\begin{equation}\label{eq:explicit-compare}
 x\,\Phi\,(\alpha+x'),\qquad
 y\,\Phi\,(y'+\beta),\qquad
 z\,\Phi\,(\alpha+z'+\beta).
\end{equation}
If the two triples are elementwise interalgebraic over $F$, then $\Phi$ is a quasi-isomorphism.
\end{theorem}

For a quasi-isomorphism $\Phi\le G\times G'$, let
\[
K_\Phi=\pi_{G'}\bigl(\ker(\pi_G|_\Phi)\bigr)\le G'.
\]
This is a finite $F$-definable subgroup.  The finite-indeterminacy quotient formalism of Boege--Yashfe~\cite[Remark~3.3]{BoegeYashfe2026} identifies $\Phi$ with a surjective $F$-definable homomorphism
\begin{equation}\label{eq:qiso-quotient}
 f_\Phi:G\longrightarrow G'/K_\Phi
\end{equation}
with finite kernel.  The comparison theorem is cited from the Boege--Yashfe preprint; together with \eqref{eq:qiso-quotient}, it supplies the finite-quotient comparison used below.

\subsection{One common quotient from the common anchor}

Recover a group $G_i$ from each cell $Q_i$.  Choose lower-case representatives so that
\begin{equation}\label{eq:cellcoords}
A\simF a_i,\quad S_i\simF b_i,\quad S_{i+1}\simF a_i+b_i,
\end{equation}
\[
X\simF u_i,\quad Y\simF u_i+a_i,\quad
R_i\simF u_i+b_i,\quad R_{i+1}\simF u_i+a_i+b_i.
\]
All cells share the same matroid triple $(X,A,Y)$.  Thus the group triples
\[
(u_i,a_i,u_i+a_i)\subseteq G_i
\]
are pairwise elementwise interalgebraic.  Fix $G_0$.  For each $i$, apply \cref{thm:compare} to the triples in $G_i$ and $G_0$.  This yields an $F$-definable quasi-isomorphism $\Phi_i\le G_i\times G_0$ and $F$-points $\lambda_i,\mu_i\in G_0(F)$ such that
\[
u_i\,\Phi_i\,(\lambda_i+u_0),\qquad
 a_i\,\Phi_i\,(a_0+\mu_i),\qquad
 u_i+a_i\,\Phi_i\,(\lambda_i+u_0+a_0+\mu_i).
\]
Let $K_i=K_{\Phi_i}\le G_0$ be the finite target-indeterminacy subgroup.  Put
\[
K=K_0+\cdots+K_{p-1}.
\]
The addition map $K_0\times\cdots\times K_{p-1}\to G_0$ has image $K$, so $K$ is finite; it is $F$-definable and therefore contained in $G_0(F)$.  Let
\[
G=G_0/K.
\]
ACF has elimination of imaginaries: definable equivalence classes can be represented inside definable sets.  Hence the quotient by a finite definable subgroup is definable; after replacing $G_0$ by its definably isomorphic algebraic group, this is the usual quotient by a finite algebraic subgroup.  Each $\Phi_i$ therefore induces a surjective homomorphism
\[
q_i:G_i\longrightarrow G
\]
with finite kernel.

Let $u,a\in G$ be the images of $u_0,a_0$.  They remain independent generics.  For $i>0$, write $\bar\lambda_i,\bar\mu_i$ for the images of $\lambda_i,\mu_i$ in $G$.  Then the explicit relations above give
\[
q_i(u_i)=u+\bar\lambda_i,\qquad q_i(a_i)=a+\bar\mu_i,
\]
and therefore $q_i(u_i+a_i)=u+a+\bar\lambda_i+\bar\mu_i$.  Put $d_i=q_i(b_i)$ (with zero translations for $i=0$).  Since $q_i$ has finite kernel, each element is interalgebraic with its image.  The literal cell equations then give
\[
q_i(a_i+b_i)=a+d_i+\bar\mu_i,\quad
q_i(u_i+b_i)=u+d_i+\bar\lambda_i,
\]
\[
q_i(u_i+a_i+b_i)=u+a+d_i+\bar\lambda_i+\bar\mu_i.
\]
Translation by an $F$-point does not change an algebraic-closure class.  Hence, for every $i$,
\begin{equation}\label{eq:anchored}
X\simF u,\quad A\simF a,\quad Y\simF u+a,
\end{equation}
\[
S_i\simF d_i,\quad S_{i+1}\simF a+d_i,
\quad R_i\simF u+d_i,\quad R_{i+1}\simF u+a+d_i.
\]
Thus all cells live, at the level of interalgebraicity classes, in one connected commutative group of dimension $t$.  The common anchor compares every cell directly with $G_0$, so no chain of quasi-isomorphisms is composed around the cycle.

\subsection{Two-translate rigidity and the cycle}

The comparison theorem has the following consequence.

\begin{lemma}[Two-translate rigidity]\label{lem:translate}
Let $G$ be a connected commutative $F$-definable group and $u$ generic over $F$.  If
\[
(u,v,u+v),\qquad (u,w,u+w)
\]
are elementwise interalgebraic group triples, then
\[
w=v+c\qquad\text{for some }c\in G(F).
\]
\end{lemma}
\begin{proof}
By \cref{thm:compare}, there is an $F$-definable quasi-isomorphism $\Phi\le G\times G$ and $F$-points $\alpha,\beta$ such that
\[
u\,\Phi\,(\alpha+u),\qquad v\,\Phi\,(w+\beta).
\]
Let
\[
K=\pi_2\bigl(\ker(\pi_1|_\Phi)\bigr)
\]
be the finite target-indeterminacy subgroup, let $\pi:G\to G/K$ be the quotient, and let $f:G\to G/K$ be the homomorphism associated with $\Phi$ by \eqref{eq:qiso-quotient}.  At the common generic $u$ we have
\[
f(u)-\pi(u)=\pi(\alpha).
\]
Hence the $F$-definable homomorphism $\delta=f-\pi$ has a fiber containing the generic $u$.  That fiber is a coset of $\ker\delta$, so $\ker\delta$ has full dimension and finite index.  Connectedness of $G$ forces $\ker\delta=G$, hence $f=\pi$.

Applying this equality at $v$ gives
\[
\pi(v)=\pi(w+\beta).
\]
Therefore $w-v\in -\beta+K$.  The finite $F$-definable subgroup $K$ consists of $F$-points because $F$ is algebraically closed.  Thus $w=v+c$ for some $c\in G(F)$.
\end{proof}

For the edge shared by $Q_i$ and $Q_{i+1}$, \eqref{eq:anchored} gives the two group triples
\[
(u,a+d_i,u+a+d_i),\qquad
(u,d_{i+1},u+d_{i+1}).
\]
They are elementwise interalgebraic because they represent the same three matroid points $(X,S_{i+1},R_{i+1})$; their pairwise dimensions are $2t$ by simplicity and \cref{lem:full}.  Therefore \cref{lem:translate} yields
\begin{equation}\label{eq:recurrence}
d_{i+1}=a+d_i+\varepsilon_i,\qquad \varepsilon_i\in G(F).
\end{equation}
At the wrap edge the next cell is literally $Q_0$, so put $d_p=d_0$.  Iterating \eqref{eq:recurrence} gives
\begin{equation}\label{eq:monodromy}
[p]a=-\sum_{i=0}^{p-1}\varepsilon_i\in G(F).
\end{equation}
This is the affine-monodromy identity.

For $p=2$, the two role-shifted cells are the same Fano plane; the extra Fano dependence is exactly the exceptional $(W,C,Y)$ triple allowed by the corrected seven-point theorem, so both recoveries remain valid.  For $p=3$, the wrap line $\{Y,S_2,R_0\}$ is exactly the required line in $Q_2$ and creates no forbidden extra collinearity in the other recovered cells.  Hence the same argument covers the small primes.

\subsection{Forcing the characteristic}

The element $a$ in \eqref{eq:monodromy} is generic over $F$.  Hence the $F$-defined fiber
\[
[p]^{-1}([p]a)
\]
contains a generic point.  It is a coset of $\ker[p]$, so $\ker[p]$ has full dimension and finite index.  Since $G$ is connected,
\begin{equation}\label{eq:pzero}
[p]G=0.
\end{equation}
Replace $G$ by its definably isomorphic algebraic group.  The differential of multiplication by $p$ at the identity is
\[
d[p]_e=p\,\mathrm{id}_{T_eG}.
\]
Equation \eqref{eq:pzero} makes this differential zero.  Since $G$ has positive dimension, its tangent space is nonzero; if the field characteristic were $q\ne p$, including characteristic zero, the scalar $p$ would be invertible and the differential could not vanish.  Therefore $\operatorname{char}K=p$.

Conversely, \eqref{eq:Mp-matrix} is a linear representation over $\mathbf F_p$.  This proves \cref{thm:intro-singleton}.

\section{A connected almost-entropic separation}\label{sec:separation}

The singleton characteristic theorem first yields a disconnected witness in the outer region of \cref{fig:venn}.  Its two summands will supply the spanning restrictions of the connected construction.

\subsection{The direct-sum consequence}

\begin{corollary}\label{cor:direct-separation}
If $p\ne q$ are primes, then
\[
M(p)\oplus M(q)\in\AlmostEnt\setminus\FAlg.
\]
\end{corollary}
\begin{proof}
Since $M(p)$ and $M(q)$ are linear over $\mathbf F_p$ and $\mathbf F_q$, respectively, they are folded algebraic and hence almost entropic by \cref{thm:falg-ae}.  The almost-entropic cone is closed under direct sums: approximate the two rank functions by entropic vectors, realize the two approximating random systems independently, and use additivity of entropy under independent products.  Thus $M(p)\oplus M(q)\in\AlmostEnt$.

On the other hand, \cref{thm:intro-singleton,prop:directsum} give
\[
\chi_{\FAlg}(M(p)\oplus M(q))=\{p\}\cap\{q\}=\varnothing,
\]
so the direct sum is not folded algebraic.
\end{proof}

Taking $(p,q)=(2,3)$ gives a $16$-element witness.  The next construction identifies three elements rather than keeping the two characteristic gadgets disjoint.

\subsection{Almost-entropic independent-set welding}

For independent points $x,y$, freely add a point $e$ on the rank-two flat $\cl(x,y)$, contract $e$, and delete $y$.  The surviving point $x$ then represents the identification of $x$ with $y$.  This is the principal-extension operation defined in \cref{subsec:m33}.  Mat\'u\v{s} proved that the almost-entropic class is closed under principal extensions and minors~\cite{Matus2018}; hence this operation preserves almost-entropicity.

\begin{proposition}[Independent-set welding]\label{prop:welding}
Let $M_1,M_2\in\AlmostEnt$ have disjoint ground sets, and choose ordered independent sets
\[
I_1=(x_1,\dots,x_k),\qquad I_2=(y_1,\dots,y_k).
\]
Starting from $M_1\oplus M_2$, identify $x_j$ with $y_j$ successively by the preceding principal-extension--contraction--deletion operation.  The resulting matroid $W$ is almost entropic and contains natural restrictions isomorphic to both $M_1$ and $M_2$.

If $I_1$ and $I_2$ are bases of two rank-$r$ matroids and $k=r$, then $r(W)=r$ and both factor restrictions are spanning.
\end{proposition}
\begin{proof}
Almost-entropicity follows from closure under direct sums, principal extensions, contractions, and deletions.  It remains to check that neither factor is damaged by the identifications.

Keep the original disjoint labels while performing the construction and, after the last deletion, retain $x_j$ as the common label.  For $S$ in the final ground set an induction on the $k$ identifications gives
\begin{equation}\label{eq:weld-rank}
 r_W(S)=
 \min_{J\subseteq[k]}
 \left(
 r_{M_1\oplus M_2}
 \bigl(S\cup\{x_j,y_j:j\in J\}\bigr)-|J|
 \right).
\end{equation}
Indeed, one identification transforms a rank function by
\[
 r(S)\longmapsto \min\{r(S),r(Sxy)-1\},
\]
and the resulting operators commute for the disjoint pairs.

If $S\subseteq E(M_1)$, then the $M_2$-contribution of the adjoined $y_j$ is $|J|$ because $I_2$ is independent.  Every term of \eqref{eq:weld-rank} is therefore at least $r_{M_1}(S)$, while $J=\varnothing$ gives equality.  Thus $M_1$ survives as a restriction.

Now let $T\subseteq E(M_2)$ and let $S$ be its image after replacing every selected $y_j$ by the common label $x_j$.  Put $I=\{j:y_j\in T\}$.  Taking $J=I$ in \eqref{eq:weld-rank} gives $r_{M_2}(T)$.  For arbitrary $J$, independence of $I_1$ contributes $|I\cup J|$, while deleting the elements $y_j$ with $j\in I\setminus J$ can lower $r_{M_2}(T)$ by at most $|I\setminus J|$.  Hence the corresponding term is at least
\[
 |I\cup J|+r_{M_2}(T)-|I\setminus J|-|J|
 =r_{M_2}(T).
\]
So $M_2$ also survives exactly.

If both chosen sets are full bases of rank $r$, the surviving $M_1$ restriction gives $r(W)\ge r$, while taking $J=[r]$ in \eqref{eq:weld-rank} for the whole ground set gives $r(W)\le r+r-r=r$.
\end{proof}

\subsection{The 13-element witness}\label{subsec:c23}

For $p=2,3$, write $M_p=M(p)$ and distinguish the nonshared labels by a subscript $p$.  In both standard representations from \eqref{eq:Mp-matrix},
\[
B_p=\{A_p,X_p,S_{0,p}\}
\]
is a basis.  Apply \cref{prop:welding} to identify
\[
(A_2,X_2,S_{0,2})\sim(A_3,X_3,S_{0,3})
\]
pointwise, and denote the three common elements by $A,X,S_0$.  Let the resulting matroid be $C_{2,3}$.  Its ground set is
\[
\begin{aligned}
E(C_{2,3})=\{&A,X,S_0,
Y_2,R_{0,2},S_{1,2},R_{1,2},\\
&Y_3,R_{0,3},S_{1,3},R_{1,3},S_{2,3},R_{2,3}\}.
\end{aligned}
\]
Thus
\begin{equation}\label{eq:C23-size}
|E(C_{2,3})|=7+9-3=13,
\qquad r(C_{2,3})=3.
\end{equation}
By construction there are spanning restrictions
\begin{equation}\label{eq:C23-restrictions}
C_{2,3}|E_2\cong M(2),
\qquad
C_{2,3}|E_3\cong M(3),
\qquad
E_2\cap E_3=\{A,X,S_0\}.
\end{equation}

\begin{theorem}\label{thm:C23}
The matroid $C_{2,3}$ is simple, rank $3$, $3$-connected, and
\[
C_{2,3}\in\AlmostEnt\setminus\FAlg.
\]
Consequently \cref{thm:intro-strict} holds.
\end{theorem}
\begin{proof}
We first check simplicity.  A pair contained in one factor has rank two because that factor survives as a simple restriction.  Consider a cross-factor pair.  In \eqref{eq:weld-rank}, the term for $J=\varnothing$ is two.  For $|J|=1$, simplicity of both factors makes the two factor contributions equal to two, so the term is three; for $|J|\ge2$, independence of the welding sets makes each factor contribution at least $|J|$, so the term is at least $|J|\ge2$.  Thus every cross-factor pair also has rank two.  By \cref{prop:welding}, $C_{2,3}$ is almost entropic and has rank $3$.

If $C_{2,3}$ were folded algebraic over a field $K$, restriction to the first spanning copy in \eqref{eq:C23-restrictions} would give a folded-algebraic representation of $M(2)$ over $K$, and restriction to the second would give one of $M(3)$ over the same field.  By \cref{thm:intro-singleton} this would force
\[
\operatorname{char}K=2
\qquad\text{and}\qquad
\operatorname{char}K=3,
\]
a contradiction.

It remains to prove $3$-connectivity.  The restriction $M(2)=F_7$ is spanning and connected, so $C_{2,3}$ is connected.  In a simple rank-three matroid, a $2$-separation $(U,V)$ with $|U|,|V|\ge2$ forces
\[
r(U)=r(V)=2.
\]
Intersecting with the spanning Fano restriction would then cover all seven Fano points by two rank-two subsets.  Each such subset is contained in a Fano line, and two Fano lines contain at most six points.  This is impossible.  Hence $C_{2,3}$ has no $2$-separation and is $3$-connected.
\end{proof}

For reference, its nontrivial rank-two flats consist of one five-point line, three four-point lines, and nine three-point lines.  The companion certificate reconstructs this distribution directly from the rank oracle.

The common points of the connected witness allow both characteristic obstructions to remain visible from one dealer.

\subsection{A single-dealer consequence}\label{subsec:dealer}

For a matroid $M$ and a nonloop $d\in E(M)$, write
\[
\Gamma_d(M)=\{Q\subseteq E(M)-d:d\in\cl_M(Q)\}
\]
for its dealer port.  We use the following exact pointed version of the standard Brickell--Davenport rigidity theorem~\cite{BrickellDavenport1991}.

A participant $x\ne d$ is \emph{essential} if there is a set $Q\subseteq E(M)-\{d,x\}$ such that $d\notin\cl_M(Q)$ but $d\in\cl_M(Qx)$.

Call $h:2^E\to\mathbb Z_{\ge0}$ an \emph{algebraic polymatroid over $K$} if there are finite tuples $X_e$ in a field extension of $K$ such that
\[
h(A)=\trdeg_K K(X_e:e\in A)\qquad(A\subseteq E).
\]
Write $h(d\mid Q)=h(Qd)-h(Q)$.

\begin{proposition}[Exact port rigidity]\label{prop:port-rigidity}
Let $M$ be connected and let $d\in E(M)$.  Suppose $\Gamma_d(M)$ has an exact $t$-fold algebraic pointed realization over a field $K$: that is, an algebraic polymatroid $h$ on $E(M)$ satisfying
\[
h(d)=t,\qquad h(x)\le t\quad(x\ne d),
\]
and
\[
h(d\mid Q)=
\begin{cases}
0,&Q\in\Gamma_d(M),\\
t,&Q\notin\Gamma_d(M).
\end{cases}
\]
Then $M$ is $t$-fold algebraic over $K$.
\end{proposition}
\begin{proof}
Because $M$ is connected, every participant is essential in its port.  Fix $x\ne d$ and choose $Q$ with $Q\notin\Gamma_d(M)$ but $Qx\in\Gamma_d(M)$.  Expanding $h(Qxd)$ in the two orders gives
\[
h(d\mid Q)+h(x\mid Qd)=h(x\mid Q)+h(d\mid Qx).
\]
The pointed conditions therefore imply
\[
t+h(x\mid Qd)=h(x\mid Q)\le h(x)\le t,
\]
so $h(x)=t$ for every participant.

Normalize $f=h/t$.  Then $f$ is a polymatroid with singleton rank one at every element and with a perfect connected dealer port.  The exact polymatroidal Brickell--Davenport statement says that a normalized polymatroid with unit singleton ranks and a perfect connected port is a matroid rank function~\cite[Theorem~4.1.2]{PadroSecretSharingNotes}; it is the combinatorial form of~\cite[Theorem~1]{BrickellDavenport1991}.  Hence $f$ is the rank function of a connected matroid $N$ with dealer port $\Gamma_d(M)$.  Lehman's port-uniqueness theorem says that a connected matroid is determined by the circuits containing one fixed element, equivalently by any one of its connected ports~\cite{Lehman1964}; see also the constructive oracle formulation in~\cite[p.~190]{CoullardHellerstein1996}.  Thus $N=M$.  Therefore $h=t r_M$, and the algebraic realization of $h$ is a $t$-fold algebraic representation of $M$.
\end{proof}

Take the common element $A$ as dealer in $C_{2,3}$ and put
\[
\Gamma=\Gamma_A(C_{2,3}).
\]
There are $12$ participants, all essential because $C_{2,3}$ is connected.  In fact the participant matroid $C_{2,3}\setminus A$ is itself $3$-connected.  It is simple of rank $3$ and contains the spanning restriction $M(2)\setminus A\cong M(K_4)$.  A $2$-separation would cover the six elements of this restriction by two sets of rank at most two.  Each such set has at most three elements, and a three-element rank-two set in $M(K_4)$ is a triangle.  Both sets would therefore have to be triangles, but any two triangles of $K_4$ share an edge and cover at most five edges, a contradiction.

Let
\[
P_2=E_2-\{A\},\qquad P_3=E_3-\{A\}.
\]
Since the two Gordon matroids are literal restrictions of $C_{2,3}$, deleting all participants outside $P_2$ or $P_3$ gives
\begin{equation}\label{eq:port-minors}
\Gamma|_{P_2}=\Gamma_A(M(2)),
\qquad
\Gamma|_{P_3}=\Gamma_A(M(3)).
\end{equation}

\begin{corollary}[One dealer retains both characteristics]\label{cor:dealer}
The $12$-participant port $\Gamma_A(C_{2,3})$ admits no exact pointed folded-algebraic realization over any field, at any folding degree.
\end{corollary}
\begin{proof}
Assume that such a realization exists over a field $K$ with folding degree $t$.  Restricting the pointed algebraic polymatroid to the dealer together with $P_2$ and $P_3$ gives exact pointed $t$-fold realizations of the two ports in \eqref{eq:port-minors}.  By \cref{prop:port-rigidity}, both $M(2)$ and $M(3)$ would then be $t$-fold algebraic over $K$.  The singleton characteristic theorem forces simultaneously $\operatorname{char}K=2$ and $\operatorname{char}K=3$, a contradiction.
\end{proof}

Thus one connected dealer port retains both incompatible characteristic requirements, and every participant is essential.

\section{Discussion}\label{sec:discussion}

Folded algebraicity gives a common finite-packet enlargement of algebraic and folded-linear representation, and every folded-algebraic matroid is almost entropic.  The Gordon matroids show, however, that finite packetization does not remove characteristic information.  Their local seven-point cells have algebraic realizations in every characteristic, but the cyclic arrangement of those cells forces a torsion identity in a positive-dimensional algebraic group.  This is what determines the characteristic.

Independent-set welding converts the resulting direct-sum separation into a connected one.  The $13$-element matroid $C_{2,3}$ has rank $3$ and is $3$-connected, while its two spanning restrictions impose incompatible folded-algebraic characteristics.  A shared element may also be chosen as dealer, giving one connected port in which all participants are essential and the same incompatibility remains visible.

For the inner classes, \cref{prop:m33} gives the explicit bounds
\[
9\le n_{\min}(\FLin\setminus\Alg)\le24,
\qquad
8\le n_{\min}(\FAlg\setminus(\Alg\cup\FLin))\le33.
\]
These are the certified intervals for the two regions.

Two questions stand out.  First, how small can a matroid in $\AlmostEnt\setminus\FAlg$ be?  The construction gives the certified interval $8\le n_{\min}\le13$.  Second, which characteristic sets can occur for folded-algebraic matroids, and how does the answer interact with connectivity and folding degree?  The family $M(p)$ supplies singleton spectra and independent-set welding combines incompatible singleton constraints; a broader structural description of possible spectra remains open.

\section*{Code and data availability}

The exact computational certificates supporting \cref{prop:mfn10,prop:m33,thm:C23,cor:dealer} are archived at Zenodo under DOI~\href{https://doi.org/10.5281/zenodo.22178786}{10.5281/zenodo.22178786}~\cite{KhazaeiCompanion2026}.  The deposit contains the coordinate functions and circuit data for $M^{\mathrm{FN}}_{10}$, the width-two $\mathbb F_5$ packet matrices for the fifteen-element core and for $A_{24}^{(5)}$, the exponent model of $NnF(5)$, three verifier scripts and their canonical JSON outputs, rank oracles for $M_{33}$ and $C_{2,3}$, and a SHA-256 manifest of every distributed file.  All arithmetic is exact: integer, finite-field, rational, or symbolic polynomial.  Running \texttt{python run\_all.py} recomputes all three certificates and compares them with the distributed files.  The archive certifies finite identities and rank calculations.  The characteristic obstructions \eqref{eq:c15-nonalg} and \eqref{eq:nnf5-noflin}, \cref{lem:principal-fixed}, and the representability conclusions of \cref{prop:m33,thm:C23,cor:dealer} are mathematical arguments rather than program outputs.

\section*{Declaration of generative AI and AI-assisted technologies in the manuscript preparation process}

During the preparation of this work, the author used ChatGPT (OpenAI) for mathematical exploration; assistance in checking arguments and developing and testing computational certificates; and drafting, restructuring, and language editing.  The author reviewed, verified, and revised all resulting material and takes full responsibility for the content of the publication.

\bibliographystyle{alpha}
\begingroup
\small
\bibliography{references}
\endgroup

\end{document}